\documentclass[a4paper,11pt]{article}
\usepackage{amssymb}
\usepackage{amsthm}
\usepackage[utf8]{inputenc}
\usepackage[T1]{fontenc}
\usepackage{lmodern}
\usepackage{graphicx}
\usepackage[a4paper]{geometry}
\usepackage{url}
\usepackage{float}
\usepackage[Rejne]{fncychap}
\usepackage{amsmath}
\usepackage[colorlinks=true,urlcolor=blue,linkcolor=blue,citecolor=red]{hyperref}
\usepackage{changepage}
\newcommand{\R}{\mathbb{R}}

\newcommand{\Conf}{\mathrm{Conf}}

\newcommand{\vect}{\mathrm{vect}}

\newcommand{\Mink}{\mathrm{Mink}}

\newcommand{\hy}{\mathbb{H}}

\newcommand{\eeu}{\widetilde{Ein}_{1,n-1}}

\begin{document}

\newtheorem{theorem}{Theorem}

\newtheorem{definition}{Definition}

\newtheorem{lemma}{Lemma}

\newtheorem{fact}{Fact}

\newtheorem{criterion}{Criterion}

\newtheorem{remark}{Remark}

\newtheorem{example}{Example}

\newtheorem{proposition}{Proposition}

\newtheorem{corollary}{Corollary}

\begin{center}
\Large \textbf{Causal completion of a globally hyperbolic conformally flat spacetime}
\end{center}

\begin{center}
Rym SMAÏ 
\end{center}

\paragraph{Abstract.} In \cite{Kronheimer}, Geroch, Kronheimer and Penrose introduced a way to attach ideal points to a spacetime $M$, defining the \emph{causal completion} of $M$. They established that this is a topological space which is Hausdorff when $M$ is globally hyperbolic. In this paper, we prove that if, in addition, $M$ is \emph{simply-connected} and \emph{conformally flat}, its causal completion is a topological manifold with boundary homeomorphic to $S \times [0,1]$ where $S$ is a Cauchy hypersurface of $M$. We also introduce three remarkable families of globally hyperbolic conformally flat spacetimes and provide a description of their causal completions.

\section{Introduction}

An important concept in Lorentzian geometry is \emph{causality}. Indeed, the tangent vectors to a Lorentzian manifold split into three classes: those of negative, positive and null norm, known as \emph{timelike, spacelike} and \emph{lightlike} vectors respectively. The curves whose tangent vectors are timelike or lightlike are called \emph{causal curves}. The causal structure describes which points of the manifold can be connected, or not, by a causal curve. Lorentzian manifolds under consideration are usually oriented and time-oriented, referred to as \emph{spacetimes}. A~time-orientation provides an orientation for every causal curve, making it either future-directed or past-directed. 

Among causal properties, special importance is given to \emph{global hyperbolicity}. This is a standard assumption for spacetimes considered as cosmological models in general relativity. According to a classical theorem by Geroch \cite{Geroch1970}, a spacetime is globally hyperbolic (abbrev. GH) if it admits a topological hypersurface that is intersected exactly once by every inextensible causal curve, known as \emph{a Cauchy hypersurface}. It turns out that Cauchy hypersurfaces of a GH spacetime are homeomorphic.

In \cite{Kronheimer}, Geroch, Kronheimer and Penrose defined how to attach to a spacetime $M$ ideal points. Their construction is entirely based on the causal structure of $M$ and formalizes the concept of \emph{endpoints at infinity} of inextensible causal curves.  Those in the future form \emph{the future causal boundary} while those in the past form \emph{the past causal boundary} of~$M$. The union of $M$, its future and its past causal boundary define \emph{the causal completion} of~$M$. The authors proved that the causal completion is a topological space in which $M$ is a dense open subset. Moreover, it is Hausdorff when $M$ is GH. It is worth noting that, in general, the topology of the causal completion can be quite intricate (see Section \ref{sec: causal completion in the conformally flat setting}). 

In this paper, we explore the causal completion of a GH \emph{conformally flat} spacetime of dimension $n \geq 3$. These are locally homogeneous manifolds, no longer equipped with a single Lorentzian metric but with \emph{a conformal class} of Lorentzian metrics. The model space of these geometric structures is the Lorentzian analogue of the conformal sphere, the so-called \emph{Einstein universe} $Ein_{1,n-1}$; its group of conformal transformations is the linear group $O(2,n)$. Notice that there is still a notion of causality in this setting since the sign of the norm of a tangent vector is invariant under conformal changes of metrics. In Section \ref{sec: causal completion in the conformally flat setting}, we prove the following result:

\begin{theorem} \label{intro: main theorem}
Let $M$ be a simply-connected conformally flat GH spacetime of dimension $n \geq 3$ without conjugate points. The causal completion of $M$ is a topological manifold with boundary, homeomorphic to $S \times [0,1]$ where $S$ is any Cauchy hypersurface of $M$.
\end{theorem}
 
The assumptions \emph{simply-connected} and \emph{without conjugate points} (see Section \ref{sec: conformally flat spacetimes} for the definition) are natural in the sense that if one of them is not satisfied, we construct easily counterexamples. In fact, Theorem \ref{intro: main theorem} still holds for the larger class of \emph{developable} spacetimes (see Section \ref{sec: conformally flat spacetimes}).

The proof of Theorem \ref{intro: main theorem} centrally relies on the notion of the \emph{enveloping space of a simply-connected GH conformally flat spacetime}, introduced in a previous paper \cite{smai2023enveloping}.
We established that $M$ can be embedded in a conformally flat spacetime \mbox{$E(M) = \mathcal{B} \times \R$,} where $\mathcal{B}$ is a conformally flat Riemannian manifold diffeomorphic to any Cauchy hypersurface of $M$. This embedding is realized as the domain bounded by the graphs of two real functions, $f^+$ and $f^-$, defined on an open subset of $\mathcal{B}$. The space $E(M)$ is referred to as \emph{an enveloping space of $M$}.
The proof of Theorem \ref{intro: main theorem} consists in proving that the graphs of $f^+$ and $f^-$ precisely correspond to the future and past causal boundary of $M$.
 
\subsection{Application: equivalence between $\mathcal{C}_0$-maximality and $\mathrm{C}$-maximality} \label{intro: application}


There is a natural partial order relation on globally hyperbolic conformal spacetimes. Given two GH conformal spacetimes $M$ and $N$, we say that $N$ is a Cauchy-extension of $M$ if there exists a conformal embedding $f$ from $M$ to $N$ sending every Cauchy hypersurface of $M$ on a Cauchy hypersurface of $N$. The map $f$ is called \emph{a conformal Cauchy-embedding}. A GH conformal spacetime $M$ is said \emph{$\mathrm{C}$-maximal} if every conformal Cauchy-embedding from $M$ to a GH conformal spacetime $N$ is surjective. Given a GH conformal spacetime $M$, the existence of a $\mathrm{C}$-maximal Cauchy-extension of $M$ is \emph{a priori} not insured. Nevertheless, it is within the category of conformally flat spacetimes. More precisely, a GH \emph{conformally flat} spacetime $M$ is said \emph{$\mathcal{C}_0$-maximal} if every conformal Cauchy-embedding from $M$ to a GH \emph{conformally flat} spacetime is surjective. A result due to C. Rossi \cite{Salvemini2013Maximal} states that any conformally flat spacetime admits a $\mathcal{C}_0$-maximal Cauchy-extension and furthermore, this extension is unique up to conformal diffeomorphism. We provided a new proof of this result, which involves the notion of enveloping space as detailed in \mbox{\cite[Sec. 5]{smai2023enveloping}.} 

A natural question arises: Is the $\mathcal{C}_0$-maximal extension also $\mathrm{C}$-maximal? \emph{A priori}, there is no reason for this to be true. However, we establish that it is indeed the case.

\begin{theorem} \label{intro: C_0-maximality implies C-maximality}
Let $M$ be a GH conformally flat spacetime. If $M$ is $\mathcal{C}_0$-maximal, it is $\mathrm{C}$-maximal.
\end{theorem}

The converse assertion is obviously true: any $\mathrm{C}$-maximal conformally flat spacetime is $\mathcal{C}_0$-maximal. Thus, Theorem \ref{intro: C_0-maximality implies C-maximality} says that the notions of $\mathcal{C}_0$-maximality and $\mathrm{C}$-maximality are equivalent for conformally flat spacetimes. Hence, a $\mathcal{C}_0$-maximal spacetime will be simply referred to as \emph{maximal} (abbrev. GHM) - while keeping in mind that it is for the ordering relation defined by \emph{conformal} Cauchy-embeddings.

This result was initially proved by Rossi in her thesis (see \cite[Chap. 7, Sec. 2]{salveminithesis}). The key idea of the proof is that the $\mathrm{C}$-maximality of a conformal spacetime~$M$ can be characterized by the points of its causal boundary (see Section \ref{sec: criterion of maximality}). When $M$ is conformally flat, we prove that the description of the causal completion of $M$ given by Theorem \ref{intro: main theorem} immediately implies Theorem \ref{intro: C_0-maximality implies C-maximality}.

\subsection{Future work: complete photons}

In \cite{witten1990structure}, Witten proposed the problem of classifying GH spacetimes of constant curvature, and more generally conformally flat spacetimes. GH spacetimes of constant curvature have been extensively studied by several authors as Mess, Scannell, Barbot, Mérigot and \mbox{Bonsante.} A direction still little investigated is the study of GH \emph{conformally flat} spacetimes. The first results on this topic are attributed to C. Rossi. A fundamental result of hers states that a GHM conformally flat spacetime whose universal cover admits conjugate points is a finite quotient of the universal cover of Einstein universe \mbox{(see \cite[Theorem 10]{Salvemini2013Maximal}).}

Rossi's result completely classifies GHM conformally flat spacetimes with conjugate points. In continuation, we investigate GHM conformally flat spacetimes \emph{without conjugate points}. An intriguing scenario arises when the universal cover of the spacetime contains \emph{complete} photons (see Definition \ref{def: complete photons}). In Section \ref{sec: examples}, we outline three notable families of such spacetimes and provide descriptions of their causal boundaries.
 
A strategy for studying globally hyperbolic (GH) conformally flat spacetimes with complete photons involves exploring the points of the causal boundary that serve as endpoints for complete photons, called \emph{complete ideal points}. In a forthcoming paper, the description of the causal boundary provided by Theorem \ref{intro: main theorem} will allow us to conduct an in-depth study of the complete ideal points.

\subsection*{Overview of the paper}

Section \ref{sec: causally convex subsets of eeu} provides an overview of basic notions in Lorentzian geometry, Einstein universe, and conformally flat Lorentzian structures. We also recall the notion of enveloping space introduced in \cite[Sec. 4]{smai2023enveloping}. In Section \ref{sec: causal completion}, we define the causal completion of a GH spacetime and we prove that it is a topological space which is Hausdorff. Section \ref{sec: causal completion in the conformally flat setting} is dedicated to the proof of Theorem \ref{intro: main theorem} while Section \ref{sec: maximality} focuses on the proof of Theorem~\ref{intro: C_0-maximality implies C-maximality}. Last but not least, we introduce in Section \ref{sec: examples} three remarkable families of globally hyperbolic maximal conformally flat spacetimes and we describe their causal completion.


\subsection*{Acknowledgement} I would like to express my gratitude to my PhD advisor for the insightful discussions that made this work possible, for his valuable remarks, and for his guidance. I am also thankful to Charles for his interest in my work and for his thoughtful comments and suggestions. This work of the Interdisciplinary Thematic Institute IRMIA++, as part of the ITI 2021-2028 program of the University of Strasbourg, CNRS and Inserm, was supported by IdEx Unistra (ANR-10-IDEX-0002), and by SFRI-STRAT’US project (ANR-20-SFRI-0012) under the framework of the French Investments for the Future Program. 

\section{Einstein universe and conformally flat spacetimes} \label{sec: conformally flat spacetimes}

This section introduces the model of conformally flat Lorentzian structures, the so-called \emph{Einstein universe}. Let us start with some basics in Lorentzian geometry.


\subsection{Preliminaries on Lorentzian geometry}

We give here a short exposition on the causality of spacetimes. We direct to \cite[Chap. 14]{oneill} for more details.

\paragraph{Spacetime.} A Lorentzian manifold is a smooth manifold of dimension $n$ equipped with a non-degenerate symmetric $2$-tensor $g$ of signature $(1,n-1)$.

In a Lorentzian manifold $(M,g)$, we say that a non-zero tangent vector $v$ is \emph{timelike, lightlike, spacelike} if $g(v,v)$ is respectively negative, zero, positive. The set of timelike vectors is the union of two convex open cones. When it is possible to make a continuous choice of a connected component in each tangent space, the manifold $M$ is said \emph{time-orientable}. The timelike vectors in the chosen component are said \emph{future-directed} while those in the other component are said \emph{past-directed}. A \emph{spacetime} is an oriented and time-oriented Lorentzian manifold. 

\paragraph{Future, past.} In a spacetime $M$, a differential curve is \emph{timelike, lightlike, spacelike} if its tangent vectors are timelike, lightlike, spacelike. It is \emph{causal} if its tangent vectors are either timelike or lightlike. 

Given a point $p$ in $M$, the \emph{future} (resp. \emph{chronological future}) of $p$, denoted $J^+(p)$ (resp. $I^+(p)$), is the set of endpoints of future-directed causal (resp. timelike) curves starting from $p$. More generally, the future (resp. chronological future) of a subset $A$ of $M$, denoted $J^+(A)$ (resp. $I^+(A)$), is the union of $J^+(a)$  (resp. $I^+(a)$) where $a \in A$. An open subset $U$ of $M$ is a spacetime and the intrinsic causality relations of $U$ imply the corresponding ones in $M$. We denote $J^+(A,U)$ (resp. $I^+(A,U)$) the future (resp. chronological future) in the manifold $U$ of a set $A \subset U$. Then, $I^+(A, U) \subset I^+(A) \cap U$. Dual to the preceding definitions are corresponding \emph{past} versions. In general, \emph{past} definitions and proofs follows from  future versions (and vice versa) by reversing time-orientation.

\paragraph{Achronal, acausal subsets.} A subset $A$ of a spacetime $M$ is called \emph{achronal} (resp. \emph{acausal}) if no timelike (resp. causal) curve intersects $A$ more than once.

A subset $A$ of $M$ is said to be \emph{edgeless} if for every $p \in A$, there exists an open neighborhood $U$ of $p$ such that:
\begin{itemize}
\item $U \cap A$ is achronal in $U$;
\item every causal curve contained in $U$ joining a point of $I^-(p,U)$ to a point of $I^+(p,U)$ intersects $U \cap A$.
\end{itemize}

\begin{example}
The subset $A = \{(0,x) \in \R^{1,1};\ x \in [0,1]\}$ of the $2$-dimensional Minkowski spacetime is not edgeless since the second condition is not satisfied at the points $(0,0)$ and $(0,1)$ as it is shown in Figure \ref{figure: edgeless}.
\end{example}

\begin{figure}[h!] 
\centering
\includegraphics[scale=1]{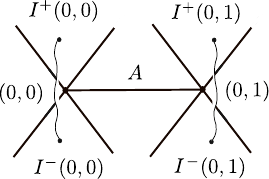}
\caption{Achronal subset of $\R^{1,1}$ which is not edgeless.} 
\label{figure: edgeless}
\end{figure}

\paragraph{Causal convexity.} In Riemannian geometry, it is often useful to consider open neighborhoods which are \emph{geodesically convex}. In Lorentzian geometry, there is, in addition, \emph{a causal convexity} notion.  A subset $U$ of $M$ is said \emph{causally convex} if for every $p, q \in U$, any causal curve of $M$ joining $p$ to $q$ is contained in $U$.

\paragraph{Global hyperbolicity.} A spacetime $M$ is said \emph{strongly causal} if for every point $p \in M$ and every neighborhood $U$ of $p$, there exists a neighborhood $V$ of $p$ contained in $U$, which is causally convex in $M$.

A spacetime $M$ is said \emph{globally hyperbolic} (abbrev. GH) if the two following conditions hold:
\begin{enumerate}
\item $M$ is strongly causal.
\item all the intersections $J^-(p) \cap J^+(q)$, where $p,q \in M$, are compact.
\end{enumerate}
By a classical theorem of Geroch \cite{Geroch1970}, a spacetime $M$ is GH if and only if it admits a topological hypersurface which is met by every inextensible causal curve exactly once. This result was improved by Bernal and Sanchez \cite{BernalSanchez}: they proved that GH spacetimes admit \emph{smooth} Cauchy hypersurfaces. It turns out that the Cauchy hypersurfaces of a GH spacetime are all homeomorphic (even diffeomorphic if they are smooth). Thus, if one of them is compact, they are all compact. In this case, the spacetime is said \emph{Cauchy-compact}.

\paragraph{Conformal spacetimes.} Two Lorentzian metrics $g$ and $g'$ on a manifold $M$ are said \emph{conformally equivalent} if there is a smooth function $f$ from $M$ to $\R$ such that $g' = e^f g$. The conformal class of $g$ is the set of Lorentzian metrics conformally equivalent to~$g$. 

Causality is a conformal notion. Indeed, given a Lorentzian manifold $(M,g)$, the type of a tangent vector to $M$ depends only the conformal class of $g$. Then, it is relevant to consider manifolds equipped with a conformal class of Lorentzian metrics.
We call \emph{conformal spacetime} an oriented smooth manifold equipped with a conformal class of Lorentzian metrics and a time-orientation. Let us point out that, in general, geodesics are not preserved by conformal changes of metrics, except lightlike geodesics as non-parametrized curves.

\subsection{Geometry of Einstein universe} \label{sec: Einstein universe}

Let $\R^{2,n}$ be the vector space $\R^{n+2}$ of dimension $(n+2)$ equipped with the nondegenerate quadratic form $q_{2,n}$ of signature $(2,n)$ given by
\begin{align*}
q_{2,n}(u,v,x_1,\ldots, x_n) &= -u^2 - v^2 + x_1^2 + \ldots + x_n^2
\end{align*}
in the coordinate system $(u,v,x_1,\ldots,x_n)$ associated to the canonical basis of $\R^{n+2}$.

\paragraph{The Klein model.} \emph{Einstein universe} of dimension $n$, denoted by $\mathsf{Ein}_{1,n-1}$, is the space of isotropic lines of $\R^{2,n}$ with respect to the quadratic form $q_{2,n}$, namely
\begin{align*}
\mathsf{Ein}_{1,n-1} &= \{[x] \in \mathbb{P}(\R^{2,n}):\ q_{2,n}(x) = 0\}.
\end{align*}
In practice, it is more convenient to work with the double cover of the Einstein universe, denoted by $Ein_{1,n-1}$:
\begin{align*}
Ein_{1,n-1} &= \{[x] \in \mathbb{S}(\R^{2,n}):\ q_{2,n}(x) = 0\}
\end{align*}
where $\mathbb{S}(\R^{2,n})$ is the sphere of rays, namely the quotient of $\R^{2,n} \backslash \{0\}$ by positive homotheties. 

\paragraph{Conformal structure.} The choice of a \emph{timelike} $2$-plane of $\R^{2,n}$, i.e. a $2$-plane on which the restriction of $q_{2,n}$ is negative definite, defines \emph{a spatio-temporal decomposition} of Einstein universe:

\begin{lemma}
Any timelike plane $P \subset \R^{2,n}$ defines a diffeomorphism between $\mathbb{S}^{n-1} \times \mathbb{S}^1$ and $Ein_{1,n-1}$.
\end{lemma}

\begin{proof}
Consider the orthogonal splitting $\R^{2,n} = P^\perp \oplus P$ and call $q_{P^\perp}$ and $q_{P}$ the positive definite quadratic form induced by $\pm q_{2,n}$ on $P^\perp$ and $P$ respectively. The restriction of the canonical projection $\R^{2,n} \backslash \{0\}$ on $\mathbb{S}(\R^{2,n})$ to the set of points $(x,y) \in P^\perp \oplus P$ such that $q_{P^\perp}(x) = q_P(y) = 1$ defines a map from $\mathbb{S}^{n-1} \times \mathbb{S}^1$ to $Ein_{1,n-1}$. It is easy to check that this map is a diffeomorphism.
\end{proof}

For every timelike plane $P \subset \R^{2,n}$, the quadratic form $q_{2,n}$ induces a Lorentzian metric $g_P$ on $\mathbb{S}^{n-1} \times \mathbb{S}^1$ given by
\begin{align*}
g_P &= d\sigma^2 (P) - d\theta^2(P)
\end{align*}
where $d\sigma^2(P)$ is the round metric on $\mathbb{S}^{n-1} \subset (P^\perp,q_{P^\perp})$ induced by $q_{P^\perp}$  and $d\theta^2(P)$ is the round metric on $\mathbb{S}^{1} \subset (P,q_P)$ induced by $q_P$.

An easy computation shows that if $P' \subset \R^{2,n}$ is another timelike plane, the Lorentzian metric $g_{P'}$ is conformally equivalent to $g_P$, i.e. $q_P$ and $g_P'$ are proportionnal by a positive smooth function on $\mathbb{S}^{n-1} \times \mathbb{S}^1$. As a result, Einstein universe is naturally equipped with a conformal class of Lorentzian metrics. Moreover, $Ein_{1,n-1} \simeq \mathbb{S}^{n-1} \times \mathbb{S}^1$ is oriented and time-oriented by the timelike vector field $\partial_\theta$. Hence, Einstein universe is a conformal spacetime. It turns out that the causal structure of Einstein universe is trivial: any point is connected to any other point by a causal curve (see e.g. \cite[Chap. 2, Cor. 2.10]{salveminithesis}). We will see that the causal structure of the universal cover of Einstein universe is more interesting.

\paragraph{Conformal group.} The subgroup $O(2,n) \subset Gl_{n+2}(\R)$ preserving $q_{2,n}$, acts conformally on $Ein_{1,n-1}$. When $n \geq 3$, the conformal group of $Ein_{1,n-1}$ is \emph{exactly} $O(2,n)$. This is a consequence of the following result, which is an extension to Einstein universe, of a classical theorem of Liouville in Euclidean conformal geometry (see e.g. \cite{francesarticle}):

\begin{theorem}\label{Liouville theorem}
Let $n \geq 3$. Any conformal transformation between two open subsets of $Ein_{1,n-1}$ is the restriction of an element of $O(2,n)$.
\end{theorem}

\paragraph{Photons, lightcones and conformal spheres.} Let us characterize some remarkable subsets of $Ein_{1,n-1}$.
\begin{enumerate}
\item \emph{A photon} is the projectivization of a totally isotropic $2$-plane of $\R^{2,n}$ (see e.g. \cite[Chap. 2, Lemme 2.12]{salveminithesis}).
\item \emph{The lightcone} of a point $\mathrm{x} \in Ein_{1,n-1}$, denoted $\mathcal{C}(\mathrm{x})$, is the intersection of $Ein_{1,n-1}$ with the projectivization of the orthogonal of a representant $x \in \R^{2,n}$ of $\mathrm{x}$ in $\R^{2,n}$. Topologically, it is a double pinched torus.
\item \emph{A conformal $(k-1)$-sphere} is a connected component of the intersection of $Ein_{1,n-1}$ with the projectivization of a Lorentzian $(k+1)$-plane of $\R^{2,n}$.
\end{enumerate}

\begin{figure}[h!]
\centering
\includegraphics[scale=1.5]{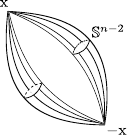}
\caption{Lightcone of $\mathrm{x} \in Ein_{1,2}$.}
\end{figure}

\paragraph{Affine charts.} For every $\mathrm{x} \in Ein_{1,n-1}$, let $M(\mathrm{x})$ denote the intersection of $Ein_{1,n-1}$ with the affine chart $A(\mathrm{x}) := \{<\mathrm{x}, .>_{2,n} < 0\}$ of $\mathbb{S}(\R^{2,n})$:
\begin{align*}
M(\mathrm{x}) &:= \{\mathrm{y} \in Ein_{1,n-1}:\ <\mathrm{x}, \mathrm{y}>_{2,n} < 0\}.
\end{align*}

\begin{definition}
We call \emph{affine chart} of $Ein_{1,n-1}$ any open subset of the form $M(\mathrm{x})$.
\end{definition}

We prove that any affine chart $M(\mathrm{x})$ of $Ein_{1,n-1}$ is naturally a conformal Minkowski spacetime. First, we define an affine structure on $M(\mathrm{x})$. This depends on the choice of a representant $x \in \R^{2,n}$ of $\mathrm{x}$ but it turns out that the class of these affine structures up to homotheties is canonical in the sense that it depends only on $\mathrm{x}$. Indeed, consider the map $f_x: M(\mathrm{x}) \times M(\mathrm{x}) \to x^\perp/\vect(x)$ given by $f_x(\mathrm{y}, \mathrm{z}) = [y - z]$ where $y$ and $z$ are representant of $\mathrm{y}$ and $\mathrm{z}$ respectively such that $<y,x>_{2,n} = <z,x>_{2,n} = -1/2$.

\begin{lemma}\label{lemma: affine chart}
The map $f_x$ defines an affine structure on $M(\mathrm{x})$ of direction $x^\perp/\vect(x)$. Moreover, if $x, x' \in \R^{2,n}$ are two distinct representant of $\mathrm{x}$, there exists $\lambda \in \R^*$ such that $f_{x'} = \lambda f_x$. \qed
\end{lemma}

The orthogonal of $x$ is degenerate. The kernel is the vector line in the direction of~$x$. A supplement of $\vect(x)$ in $x^\perp$ is a subspace such that the restriction of $q_{2,n}$ is of signature $(1,n-1)$. Therefore, $q_{2,n}$ induces a quadratic form of signature $(1,n-1)$ on the quotient $x^\perp/\vect(x)$. Hence, it results from Lemma \ref{lemma: affine chart} the following statement.

\begin{proposition}\label{prop: affine chart}
The affine chart $M(\mathrm{x})$ is a conformal Minkowski spacetime. \qed
\end{proposition}

The boundary of $M(\mathrm{x})$ in $Ein_{1,n-1}$ is the lightcone of $\mathrm{x}$. Notice that the complement of the lightcone $\mathcal{C}(\mathrm{x})$ in $Ein_{1,n-1}$ is the disjoint union of the affine charts $M(\mathrm{x})$ and $M(-\mathrm{x})$.\\

\begin{lemma}\label{lemma: euclidean hyperplanes in an affine chart}
The intersection of the affine chart $M(\mathrm{x})$ with a conformal $(k - 1)$-sphere of $Ein_{1,n-1}$ going through $\mathrm{x}$ is a spacelike $(k - 1)$-plane of $M(\mathrm{x})$.
\end{lemma}

\begin{proof}
Let $\mathrm{S}$ be a conformal $(k - 1)$-sphere going through $\mathrm{x}$. It is the intersection of $Ein_{1,n-1}$ with the projectivization of a Lorentzian $(k + 1)$-plane $P$ of $\R^{2,n}$ containing $\vect(x)$ where $x \in \R^{2,n}$ is a representant of $\mathrm{x}$. It is easy to check that the restriction of $f_x$ to $\mathrm{S} \cap M(\mathrm{x})$ defines an affine structure with direction $x^{\perp_P} / \vect(x)$ where $x^{\perp_P}$ denotes the orthogonal of $x$ in $P$. Since $P$ is Lorentzian, the restriction  of $q_{2,n}$ to $P$ induces a positive definite quadratic form on $x^{\perp_P} / \vect(x)$. The lemma follows.
\end{proof}

\paragraph{Penrose boundary.} Let $M(\mathrm{x})$ be an affine chart of $Ein_{1,n-1}$. 

\begin{definition} 
The regular part of the lightcone of $\mathrm{x}$ is called \emph{the Penrose boundary of the affine chart $M(\mathrm{x})$} and is denoted $\mathcal{J}(\mathrm{x})$.
\end{definition}

The Penrose boundary of $M(\mathrm{x})$ is the union of two connected components $\mathcal{J}^+(\mathrm{x})$ and $\mathcal{J}^-(\mathrm{x})$ where
\begin{itemize}
\item $\mathcal{J}^+(\mathrm{x})$ fibers trivially over the sphere $\mathrm{S}^+(\mathrm{x})$ of future lightlike directions at $\mathrm{x}$;
\item $\mathcal{J}^-(\mathrm{x})$ fibers trivially over the sphere $\mathrm{S}^-(\mathrm{x})$ of past lightlike directions at $\mathrm{x}$;
\end{itemize}
The fiber over a direction $[v] \in \mathrm{S}^\pm(\mathrm{x})$ is the lightlike geodesic contained in $\mathcal{J}^\pm(\mathrm{x})$ tangent to $v$ at $\mathrm{x}$.

\begin{figure}[h!]
\centering
\includegraphics[scale=1.5]{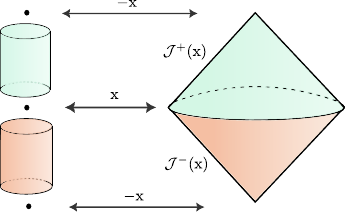}
\caption{Penrose diagramm of an affine chart $M(\mathrm{x})$ of $Ein_{1,2}$: the interior of the diamond represents $M(\mathrm{x})$; the upper and lower cones represent $\mathcal{J}^+(\mathrm{x})$ and $\mathcal{J}^-(\mathrm{x})$ \mbox{respectively;} the equatorial circle is identified to $\mathrm{x}$ and the vertices are identifies to~$-\mathrm{x}$.}
\end{figure}

\paragraph{Penrose boundary and degenerate affine hyperplanes of an affine chart.} Now, we prove that each connected component of $\mathcal{J}(\mathrm{x})$ is in bijection with the space of degenerate affine hyperplanes of $M(\mathrm{x})$. We write it for $\mathcal{J}^+(\mathrm{x})$ but of course it is similar for $\mathcal{J}^-(\mathrm{x})$.

\begin{lemma}
The intersection of the affine chart $M(\mathrm{x})$ with the lightcone of a point $\mathrm{y} \in \mathcal{J}^+(\mathrm{x})$ is a degenerate affine hyperplan of $M(\mathrm{x})$.
\end{lemma}

\begin{proof}
Let $x, y \in \R^{2,n}$ be two representants of $\mathrm{x}$ and $\mathrm{y}$ respectively. On the one hand, $<y,x>_{2,n} = 0$, i.e. $\vect(y) \subset x^\perp$. On the other hand, since $\mathrm{y} \not \in \{\mathrm{x}, -\mathrm{x}\}$, the lightlike line $\vect(y)$ is transverse to $\vect(x)$. Therefore, the projection $[y]$ of $y$ in the quotient $x^\perp/\vect(x)$ is a non-trivial isotropic vector. It is easy to check that the map $f_x$ defined above induces on the intersection of $M(\mathrm{x})$ with the lightcone of $\mathrm{y}$ an affine structure of direction the orthogonal of $[y]$ in $x^\perp/\vect(x)$. The lemma follows.
\end{proof}

\begin{lemma} \label{lemma: Penrose boundary and lightlike directions of Minkowski spacetime}
Every lightlike geodesic of the Penrose boundary $\mathcal{J}^+(\mathrm{x})$ defines a unique lightlike direction of the affine chart $M(\mathrm{x})$ and vice versa.
\end{lemma}

\begin{proof}
Let $x \in \R^{2,n}$ be a representant of $\mathrm{x}$. Recall that the vector space associated to the affine chart $M(\mathrm{x})$ is  $x^\perp/\vect(x)$. A lightlike geodesic of $\mathcal{J}^+(\mathrm{x})$ is the intersection of $Ein_{1,n-1}$ with the projectivization of a totally isotropic $2$-plane containing $x$. Therefore, a lightlike geodesic of $\mathcal{J}^+(\mathrm{x})$ is equivalent to the data of an isotropic vector in $x^\perp$ transverse to $\vect(x)$, in other words the data of an isotropic vector of $x^\perp/\vect(x)$.
\end{proof}

\begin{lemma}\label{lemma: section of Penrose boundary}
The intersection of the Penrose boundary $\mathcal{J}^+(\mathrm{x})$ with the lightcone of a point of the affine chart $M(\mathrm{x})$ is a section of the trivial fiber bundle $\mathcal{J}^+(\mathrm{x}) \to \mathrm{S}^+(\mathrm{x})$.
\end{lemma}

\begin{proof}
Let $\mathrm{x}_0 \in M(\mathrm{x})$. Let $x, x_0 \in \R^{2,n}$ be two representant of $\mathrm{x}$ and $\mathrm{x}_0$ respectively. The intersection of $\mathcal{J}^+(\mathrm{x})$ with the lightcone of $\mathrm{x}_0$ is a connected component of the intersection of $Ein_{1,n-1}$ with the projectivization of $x^\perp \cap x_0^\perp$. Notice that $x^\perp \cap x_0^\perp = \vect(x,x_0)^\perp$. Since $<x,x_0>_{2,n} < 0$, the subspace $\vect(x,x_0)$ is of type $(1,1)$. Then, $\vect(v,v_0)^\perp$ is of type $(1,n-1)$. It follows that the intersection of $\mathcal{J}^+(\mathrm{x})$ with the lightcone of $\mathrm{x}_0$ is a conformal $(n-2)$-sphere that meets every lightlike geodesic of $\mathcal{J}^+(\mathrm{x})$. The lemma follows.
\end{proof}

Let $f$ be the map which associates to every point $\mathrm{y} \in \mathcal{J}^+(\mathrm{x})$ the intersection of the lightcone of $\mathrm{y}$ with the affine chart $M(\mathrm{x})$.

\begin{proposition}\label{prop: Penrose bound. and degenerate hyperplans}
The map $f$ is a natural bijection between $\mathcal{J}^+(\mathrm{x})$ and the space of degenerate affine hyperplans of the affine chart $M(\mathrm{x})$.
\end{proposition}

\begin{proof}
We construct the inverse of $L$. Let $P$ be a degenerate affine hyperplane of $M(\mathrm{x})$. It is directed by the orthogonal of a lightlike direction of $M(\mathrm{x})$. By Lemma \ref{lemma: Penrose boundary and lightlike directions of Minkowski spacetime}, to this lightlike direction corresponds a unique lightlike geodesic $\varphi$ of $\mathcal{J}^+(\mathrm{x})$. Let $\mathrm{x}_0 \in P$. By Lemma \ref{lemma: section of Penrose boundary}, the intersection of the lightcone of $\mathrm{x}_0$ with $\mathcal{J}^+(\mathrm{x})$ meets every lightlike geodesic of $\mathcal{J}^+(\mathrm{x})$ in a unique point, in particular it meets $\varphi$ in a unique point $\mathrm{p}$. We call $g$ the map which sends $P$ on $\mathrm{p}$. It is easy to check that $g = f^{-1}$.
\end{proof}

\paragraph{Universal Einstein universe.} \emph{The universal Einstein universe} is the cyclic cover $\mathbb{S}^{n-1} \times \R$ equipped with the conformal class of $d\sigma^2 - dt^2$. Notice that in dimension $n \geq 3$, it is the universal cover of Einstein universe but this is not true in dimension $n = 2$. The timelike vector $\partial_t$ defines a time-orientation on $\eeu$. Thus, $\eeu$ is a conformal spacetime.

\subparagraph{Conjugate points.} Let $\pi: \eeu \to Ein_{1,n-1}$ and $\bar{\pi}: \eeu \to \mathsf{Ein}_{1,n-1}$ be the cyclic covering maps. We denote by $\delta$ (resp. $\sigma$) : $\eeu \to \eeu$ the generator of the Galois group of $\pi$ (resp. $\bar{\pi}$) defined by $\delta(x,t) = (x, t + 2\pi)$ (resp. $\sigma(x,t) = (-x, t + \pi)$).

\begin{definition}
Two points $p$ and $q$ of $\eeu$ are said \emph{conjugate} if $q = \sigma(p)$.
\end{definition}

\begin{remark}
If $p, q \in \eeu$ are conjugate, then $\pi(p) = - \pi(q)$.
\end{remark}

Unlike Einstein universe, the universal Einstein universe has a rich causal structure. We give a brief description of its causal structure below. We direct to \cite[Chap. 2]{salveminithesis} and \cite[Sec. 2]{Salvemini2013Maximal} for the proofs.

\subparagraph{Causal curves.} Causal curves of $\eeu$ are, up to reparametrization, the curves $(x(t), t)$ where $x: I \subset \R \to \mathbb{S}^{n-1}$ is a $1$-Lipschitz curve on the sphere defined on an interval $I$ of $\R$. In particulat, lightlike geodesics are the causal curves for which $x: I \to \mathbb{S}^{n-1}$ is a geodesic of the sphere (see e.g. \cite[Lemma 5]{Salvemini2013Maximal}). Inextensible causal curves are those for which $I = \R$ in the previous parametrization. It is then easy to see that the inextensible lightlike geodesics of $\eeu$ going through a point $(x_0, t_0)$ have common intersections at the points $\sigma^k(x_0,t_0)$, for $k \in \mathbb{Z}$; and are pairwise disjoint outside these points. This description of inextensible causal curves shows that any sphere $\mathbb{S}^{n-1} \times \{t\}$, where $t \in \R$, is a Cauchy hypersurface. Hence, $\eeu$ is globally hyperbolic.

\subparagraph{Lightcone, future and past.}  
\begin{itemize}
\item The lightcone of a point $(x_0, t_0)$ is the set of points $(x,t)$ such that $d(x,x_0) = |t - t_0|$ where $d$ is the distance on the sphere $\mathbb{S}^{n-1}$ induced by the round metric.
\item The chronological future of $(x_0, t_0)$: this is the set of points $(x,t)$ of $\mathbb{S}^{n-1} \times \R$ such that $d(x,x_0) < t - t_0$.
\item The chronological past of $(x_0, t_0)$: this is the set of points $(x,t)$ of $\mathbb{S}^{n-1} \times \R$ such that $d(x,x_0) < t_0 - t$.
\end{itemize}

\paragraph{Achronal sets.} Every achronal subset of $\eeu$ is the graph of a $1$-Lipschitz real-valued function $f$ defined on a subset of $\mathbb{S}^{n-1}$. Achronal embedded topological hypersurfaces of $\eeu$ are exactly the graphs of $1$-Lipschitz real-valued functions defined on $\mathbb{S}^{n-1}$.

Although there is no achronal subsets in $Ein_{1,n-1}$, we can keep track of the notion of achronality in $Ein_{1,n-1}$.

\begin{proposition} \label{prop: achronality in Ein}
Two distinct points $\mathrm{x}$ and $\mathrm{y}$ of $Ein_{1,n-1}$ can be lifted to points $p$ and $q$ of $\eeu$ which are not extremities of a causal curve if and only if the sign of $<\mathrm{x}, \mathrm{y}>_{2,n}$ is negative.
\end{proposition}

\paragraph{Affine charts of the universal Einstein universe.} For every $p \in \eeu$, let $\Mink_0(p)$ denote the set of points which are not causally related to $p$. Notice that $\Mink_0(p)$ is the interior of the diamond $J(\sigma(p), \sigma^{-1}(p))$. It is easy to see that this last one does not contain conjugate points. Hence, the restriction of the covering map $\pi: \eeu \to Ein_{1,n-1}$ to $\Mink_0(p)$ is injective. Moreover, using Proposition \ref{prop: achronality in Ein}, one can prove that the image of this restriction is exactly the affine chart $M(\mathrm{x})$ where $\mathrm{x} = \pi(p)$. This motivates the following definition.

\begin{definition}
We call \emph{affine chart} of $\eeu$ any open subset of the form $\Mink_0(p)$.
\end{definition}

Besides $\Mink_0(p)$, the point $p$ defines two other affines charts (see Figure \ref{figure: conjugate-points}):
\begin{itemize}
\item the set of points non-causally related to $\sigma(p)$, contained in the chronological future of $p$, denoted $\Mink_+(p)$;
\item the set of points non-causally related to $\sigma^{-1}(p)$, contained in the chronological past of $p$, denoted $\Mink_-(p)$. 
\end{itemize}

\begin{figure}[h!]
\centering
\includegraphics[scale=0.8]{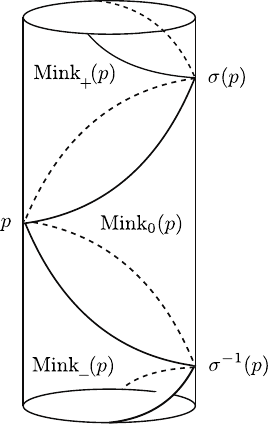}
\caption{Affine charts defined by a point $p \in \eeu$ for $n = 2$.}
\label{figure: conjugate-points}
\end{figure}

The boundary of $\Mink_0(p)$ is the union of $\partial I^+(p)$ and $\partial I^-(p)$. The regular parts of $\partial I^+(p)$ and $\partial I^-(p)$ are called \emph{the future} and \emph{the past Penrose boundary of $\Mink_0(p)$} and are denoted $\mathcal{J}^+(p)$ and $\mathcal{J}^-(p)$. 

\begin{remark}
If $\mathrm{x} = \pi(p)$, the restriction of $\pi$ to $\mathcal{J}^\pm(p)$ is injective and its image is exactly $\mathcal{J}^\pm(\mathrm{x})$.
\end{remark}

Given a point $q \in \mathcal{J}^+(p)$, the intersection of the past lightcone of $q$ with $\Mink_0(p)$ is a degenerate hyperplane $H(q)$. It follows that:
\begin{itemize}
\item the intersection of $I^-(q)$ with $\Mink_0(p)$ is the chronological past of $H(q)$ in $\Mink_0(p)$;
\item the complement of $I^-(p)$ in $\Mink_0(p)$ is the chronological future of $H(q)$ in $\Mink_0(p)$.
\end{itemize}
Notice that $I^+(q)$ is disjoint from $\Mink_0(p)$. Hence, the chronological future of $H(q)$ is exactly the set of points of $\Mink_0(p)$ which are not causally related to $q$.

\paragraph{Regular domains.} Penrose boundary is closely related to the notion of \emph{regular domains}. We adopt here the definition of \cite[Sec. 2, p. 7]{Bonsante2015SpacelikeCS}.

\begin{definition}
\emph{A future-regular domain} is a non-empty convex open domain of Minkowski spacetime obtained as the intersection of strict future half-spaces bounded by a degenerate hyperplane.
\end{definition}

We define similarly past-regular domains by reversing the time-orientation.

Let $\Lambda \subset \mathcal{J}^+(p)$. It defines naturally a convex domain in both affine charts $\Mink_0(p)$ and $\Mink_+(p)$: 
\begin{itemize}
\item the set $\Omega^+(\Lambda)$ of points of $\Mink_0(p)$ non-causally related to any point of $\Lambda$;
\item the set $\Omega^-(\Lambda)$ of points of $\Mink_+(p)$ non-causally related to any point of $\Lambda$.
\end{itemize}
The set $\Omega^+(\Lambda)$ corresponds to the intersection of the strict future half-spaces of $\Mink_0(p)$ bounded by a degenerate hyperplane of $\Lambda$ and so is a \emph{future} convex set. Similarly, $\Omega^-(\Lambda)$ is the intersection of the strict past half-spaces of $\Mink_+(p)$ bounded by a degenerate hyperplane of $\Lambda$ and so is a \emph{past} convex set. Indeed, notice that $\mathcal{J}^+(p) = \mathcal{J}^-(\sigma(p))$ is the past Penrose boundary of the affine chart $\Mink_0(\sigma(p)) = \Mink_+(p)$. Hence, $\Omega^\pm(\Lambda)$ are regular domains if and only if they are non-empty and open. In \cite[Cor. 4.11]{barbot}, the author shows that $\Omega^+(\Lambda)$ and $\Omega^-(\Lambda)$ are regular if and only if $\Lambda$ is compact. 

\begin{remark}
Set $\mathrm{x} := \pi(p)$. The restriction of $\pi$ to $\Omega^+(\Lambda)$ is injective (since it is contained in the affine chart $\Mink_0(p)$) and, by Proposition \ref{prop: achronality in Ein}, its image is the future convex domain of $M(\mathrm{x})$ defined as 
\begin{align*}
\pi(\Omega^+(\Lambda)) &= \{\mathrm{y} \in M(\mathrm{x}):\ <\mathrm{y},\mathrm{y}_0>_{2,n} < 0,\ \forall \mathrm{y}_0 \in \pi(\Lambda)\}.
\end{align*}
Similarly, the restriction of $\pi$ to $\Omega^-(\Lambda)$ is injective and its image is the past convex domain of $M(-\mathrm{x})$ defined as
\begin{align*}
\pi(\Omega^-(\Lambda)) &= \{\mathrm{y} \in M(-\mathrm{x}):\ <\mathrm{y},\mathrm{y}_0>_{2,n} < 0,\ \forall \mathrm{y}_0 \in \pi(\Lambda)\}.
\end{align*}
\end{remark} 

\begin{example}[Misner domains]
Regular domains defined by a conformal $(k - 1)$-sphere of the Penrose boundary are remarkable: they can be described as the chronological future/past of a spacelike $(n - k - 1)$-plane of Minkowski spacetime. These domains appeared naturally in the study of GHCM \emph{flat} spacetimes (see \cite[Sec. 3.2]{barbot}) and are called \emph{Misner domains}\footnote{These spacetimes have been called after the mathematician Charles W. Misner since they can be seen as a generalization of the two-dimensional spacetime described by Misner in \cite{misner1967taub}, namely the quotient by a boost of a half space of $\R^{1,1}$ bounded by a lightlike straight line.}.
\end{example}

\subsection{Conformally flat spacetimes} \label{sec: conformally flat spacetimes}

A spacetime is said \emph{conformally flat} if it is locally conformal to Minkowski spacetime. In dimension $n \geq 3$, by Liouville theorem, a spacetime is conformally flat if and only if it is equipped with $(G,X)$-structure where $X = \eeu$ and \mbox{$G = \Conf(\eeu)$} is its group of conformal transformation.
Therefore, a conformally flat Lorentzian structure on a manifold $M$ of dimension $n \geq 3$ is encoded by the data of a development pair $(D,\rho)$ where $D: \tilde{M} \to \eeu$ is a local diffeomorphism called \emph{developing map} and $\rho: \pi_1(M) \to \Conf(\eeu)$ is the associated \emph{holonomy morphism}~\footnote{We direct the reader not familiar with $(G,X)$-structures to \cite[Chapter 5]{goldman}.}. Let us make some remarks and introduce some vocabulary:
\begin{itemize}
\item In general, a developing map is only a local diffeomorphism, neither injective nor surjective. When $D$ is a global diffeomorphism, we say that the conformally flat Lorentzian structure on $M$ is \emph{complete}. 
\item A conformally flat spacetime $M$ is said \emph{developable} if any developing map descends to the quotient, giving a local diffeomorphism from $M$ to $\eeu$.
\item Two points $p, q$ of a developable conformally flat spacetime $M$ are said to be \emph{conjugate} if their images under a developing map are conjugate in $\eeu$.
\end{itemize} 

\paragraph{Enveloping space of a developable GH conformally flat spacetime.} Let $M$ be a developable GH conformally flat spacetime. In \cite[Section 4.2.]{smai2023enveloping}, we constructed a developable conformally flat spacetime $E(M)$ with the following properties:
\begin{itemize}
\item $E(M)$ fibers trivially over a conformally flat Riemannian manifold $\mathcal{B}$, diffeomorphic to a Cauchy hypersurface of $M$;
\item $M$ embeds conformally in $E(M)$ as a causally convex open subset;
\item all the conformally flat Cauchy-extensions of $M$ embeds conformally in $E(M)$ as causally convex open subsets. In particular, the $\mathcal{C}_0$-maximal extension of $M$ is the Cauchy development of a Cauchy hypersurface of $M$ in $E(M)$.
\end{itemize}
Such a spacetime $E(M)$ is called \emph{an enveloping space of $M$}.

\section{Causal completion of GH spacetimes} \label{sec: causal completion}

This section introduce the notion of \emph{causal boundary of a spacetime}, due to Geroch-Kronheimer-Penrose \cite{Kronheimer}, in the setting of GH spacetimes. Let $M$ denote a GH spacetime.

\subsection{IPs and IFs}

Let $U \subset M$ be an open subset. 

\begin{definition}
We say that $U$ is \emph{a past set} if $I^-(U) \subset U$.
\end{definition} 

The first obvious examples of past sets are chronological pasts of points, and more generally, chronological pasts of causal curves. 

\begin{definition}
We say that $U$ is \emph{an indecomposable past set} (abbrev. IP) is $U$ is a past set which can not be written as the union of two \emph{distinct} past open subsets of $M$. 
\end{definition} 

\begin{lemma} \label{lemma: sufficient condition to be an IP}
Let $U$ be a past open subset of $M$. Suppose that for every $p, q \in U$ the intersection of the chronological futures of $p$ and $q$ in $U$ is non-empty. Then $U$ is an IP.
\end{lemma}

\begin{proof}
Suppose $U = V \cup W$ where $V$ and $W$ are two distinct past open subsets of~$M$. Let $v \in V \backslash W$ and $w \in W \backslash V$. Let $u \in I^+(v) \cap I^+(w) \cap U$. Then, $v, w  \in I^-(u)$. Suppose $u \in V$. Since $V$ is a past set, $I^-(u) \subset V$; hence $w \in V$. Contradiction. Similarly, if $u \in W$, we obtain $v \in I^-(u) \subset W$. Contradiction.
\end{proof}

It follows immediately from Lemma \ref{lemma: sufficient condition to be an IP} the following statement.

\begin{corollary} \label{cor: examples of IPs}
The chronological past of a causal curve of $M$ is an IP. \qed
\end{corollary}

In \cite{Kronheimer}, the authors proved that conversely, any IP is the chronological past of some \emph{timelike} curve of $M$ (see \cite[Theorem 2.1]{Kronheimer}). Their proof is based on the fact that the condition on $U$ given in Lemma \ref{lemma: sufficient condition to be an IP} is not only sufficient but also necessary. 

The IFs are defined similarly for the reverse time-orientation. All the results stated above are true for the reverse time-orientation.

\paragraph{PIPs and TIPs.} Let $P = I^-(\gamma)$ be an IP, where $\gamma$ is a causal curve of $M$. We distinguish two cases:
\begin{enumerate}
\item The curve $\gamma$ admits a future endpoint $p$ in $M$. Then, $P$ equals $I^-(p)$. In this case, $P$ is called \emph{a proper indecomposable past set} (abbrev. PIP).
\item The curve $\gamma$ is inextendible in the future. In this case, $P$ is called \emph{a terminal indecomposable past set} (abbrev. TIP).
\end{enumerate}

Similarly, the IFs split in two classes: the PIFs, namely the chronological futures of points, and the TIFs $I^+(\gamma)$ where $\gamma$ is a causal curve inextendible in the past.

\begin{definition}
We call $\hat{M}$ (resp. $\check{M}$) the set of IPs (resp. IFs).
\end{definition}

\paragraph{Maximal TIPs.} The inclusion defines a partial ordering relation on the set of TIPs. We say that a TIP is \emph{maximal} if it is maximal for this ordering relation.

\begin{proposition}
The set of TIPs admits at least a maximal element.
\end{proposition}

\begin{proof}
We use Zorn's lemma which states that any partially ordered set containing upper bounds for every chain necessarily contains at least one maximal element.

Let $C$ a chain of TIPs, i.e. a totally ordered subset of TIPs. Let $P$ denote the union of the elements of $C$. We prove that $P$ is a TIP. Clearly, $P$ is a past-set as union of past-sets. Moreover, $P$ is indecomposable. Indeed, if $P$ is the union of two distinct past-sets $Q$ and $R$ then any TIP $P_i$ of $C$ is the union of the $P_i \cap Q$ and $P_i \cap R$. Hence, $P_i = P_i \cap Q$ or $P_i = P_i \cap R$; in other words $P_i \subset Q$ or $P_i \subset R$ for every TIP $P_i$ of $C$. It~follows that $P = Q$ or $P = R$. Lastly, suppose that $P$ is a PIP, i.e. $P = I^-(p)$ with $p \in M$. Let $P_i$ a TIP of $C$. There exists a future-inextensible causal curve $\gamma_i$ such that $P_i = I^-(\gamma_i)$. Since $P_i \subset P$, we have $I^-(\gamma_i(t)) \subset I^-(p)$ for every $t \geq 0$. Hence, $\gamma_i(t) \in J^-(p)$ for every $t \geq 0$. Contradiction.
\end{proof}

\paragraph{Causal completion of $M$.}

Since $M$ is globally hyperbolic, it is in particular \emph{past-} and \emph{future-distinguishing} (see \cite[Remark 3.23]{minguzzi2008causal}), that is the maps $p \in M \mapsto I^\pm(p)$ are injective. Therefore, the set of PIPs (PIFs) identifies with $M$. The TIPs (resp. TIFs) can be seen as the future (resp. past) endpoints at infinity of inextendible causal curves of $M$ called \emph{ideal points} in \cite{Kronheimer}. The future ideal points form \emph{the future causal boundary} of $M$ while the past ideal points form the \emph{past causal boundary} of $M$. The disjoint union of $M$ and its future and past causal boundaries is called \emph{the causal completion} of $M$, denoted by $M^\sharp$. In other words, $M^\sharp$ is the quotient of the disjoint union of the set $\hat{M}$ of IPs and the set $\check{M}$ of IFs by the equivalence relation $I^+(p) \sim I^-(p)$.

\begin{example}[Universal Einstein universe]
The description of inextendible causal curves of the universal Einstein universe shows that any TIP and any TIF is equal to the whole space. In other words, the future and the past causal boundary of the universal Einstein universe are both reduced to a single point.
\end{example}

\begin{example}[Minkowski spacetime]
In Minkowski spacetime, there are remarkable TIPs (resp. TIFs) which are the chronological futures (resp. pasts) of causal straight lines: the chronological future (resp. past) of a \emph{timelike} straight line is the whole space and the chronological future (resp. past) of a \emph{lightlike} straight line is equal to chronological future (resp. past) of the unique degenerate hyperplane containing this line (see e.g. \cite[Chapter 1, Lemme 1.2.3]{Smai2022}). It turns out that these are the only TIPs and TIFs of Minkowski spacetime (see e.g. \cite[Annexe C.1, Exemple 2]{Smai2022}). 

The causal boundary of Minkowski spacetime can be interpreted as the conformal boundary of Minkowski spacetime (see Section \ref{sec: black holes}). 
\end{example}

\begin{figure}[h!] 
\centering
\includegraphics[scale=1.4]{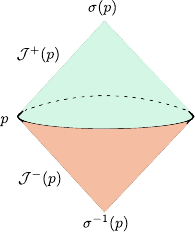}
\caption{Causal boundary of an affine chart $\Mink_0(p)$ in $\widetilde{Ein}_{1,2}$: the future causal boundary is the union of $\mathcal{J}^+(p)$ and $\sigma(p)$ and the past causal boundary is the union of $\mathcal{J}^-(p)$ and $\sigma^{-1}(p)$.} 
\label{figure: penrose}
\end{figure}

\begin{remark}
Minkowski spacetime $\Mink_0(p)$ admits one maximal TIP, $\sigma(p)$ and one maximal TIF, $\sigma^{-1}(p)$.
\end{remark}

\subsection{Topology on the causal completion}

We show here that the topology of $M$ extends naturally to the causal completion of $M$. We start with defining a topology on the set $\hat{M}$ of IPs. Since $M$ is globally hyperbolic, the topology of $M$ coincide with Alexandrov topology, namely the topology generated by the open subsets of the form $I^+(p)$, $I^-(p)$, $M \backslash J^+(p)$, $M \backslash J^-(p)$ where $p$ is a point of $M$. In other words, any open subset of $M$ is a union of finite intersections of the previous subsets. We extend this base of topology on $M$ to a base of topology on $\hat{M}$. We define two families of subsets of $\hat{M}$:
\begin{enumerate}
\item The first family is indexed by open subsets of $M$ of the form $I^+(p)$ or $M \backslash J^-(p)$ where $p$ is a point of $M$. For any open subset $U$ of this form, we call $\mathcal{O}_U$ the subset of $\hat{M}$ consisting in IPs $P$ such that $P \cap U \not = \emptyset$. 
\item The second family is indexed by open subsets of $M$ of the form $I^-(p)$ or $M \backslash J^+(p)$. For any open subset $U$ of this form, we call $\mathcal{O}'_U$ the subset of $\hat{M}$ consisting in IPs $P$ such that $\bar{P} \subset U$. 
\end{enumerate}

\begin{remark}\label{remark: induced topology on M}
The intersection of the set of PIPs with a subset $\mathcal{O}_U$ or $\mathcal{O}'_U$ is in bijection with $U$. 
\end{remark}

\begin{remark}\label{remark: open neighborhoods of a TIP}
\begin{enumerate}
\item A subset of the form $\mathcal{O}'_U$ where $U = I^-(p)$ does not contain any TIP. Indeed, suppose there exists a TIP $P$ contained in $\mathcal{O}'_U$, i.e. such that $\bar{P} \subset I^-(p)$. The TIP $P$ is the chronological past of some future causal curve $\gamma$ inextendible in the future. Then, $\gamma$ would be confined in $J^-(p)$. Contradiction.
\item A subset of the form $\mathcal{O}_U$ where $U = M \backslash J^-(p)$ contains all the TIPs of $M$. Indeed, let $P$ be a TIP. Suppose $P \not \in \mathcal{O}_U$. Then, $P \cap (M \backslash J^-(p)) = \emptyset$, equivalently $P \subset J^-(p)$. Contradiction.
\end{enumerate}
\end{remark}

Let $\tau$ be the topology on $\hat{M}$ generated by the subsets of the form $\mathcal{O}_U$ and $\mathcal{O}'_U$. It follows from Remark \ref{remark: induced topology on M} the following statement.

\begin{lemma}\label{lemma: M is homeomorphic to the PIPs}
The map from $M$ to $\hat{M}$ which sends a point $p \in M$ on the PIP $I^-(p)$ is a topological embedding. \qed
\end{lemma}

\begin{proposition}\label{prop: topology on IPs is Hausdorff}
The topology $\tau$ is Hausdorff.
\end{proposition}

\begin{proof}
Let $U$ and $V$ two distincts IPs. Then, $V \backslash U$ or $U \backslash V$ is non-empty. We suppose without loss of generality that $V \backslash U$ is non-empty. We prove that consequently, $V \backslash \bar{U}$ is non-empty. Suppose $V \backslash \bar{U} = \emptyset$, i.e. $V \subset \bar{U}$. Since $U$ is a past set, the interior of $\bar{U}$ is exactly $U$ \footnote{This follows immediately from the fact that the boundary of a past set is a closed achronal topological hypersurface of $M$ (see \cite[Corollary 27, p. 415]{oneill}).}. Then, $V \subset U$, i.e. $V \backslash U = \emptyset$. Contradiction. Hence, $V \backslash \bar{U} \not = \emptyset$. Let $w \in V \backslash \bar{U}$. Set $W = I^+(w)$ and $W' = M \backslash J^+(w)$. Then, $\mathcal{O}'_{W'}$ and $\mathcal{O}_W$ are two disjoint neighborhoods of $U$ and $V$ respectively.
\end{proof}

\begin{proposition}\label{prop: PIPs are dense in IPs}
The subspace of $\hat{M}$ consisting in PIPs of $M$ is an open subset of $\hat{M}$ dense in $\hat{M}$.
\end{proposition}

\begin{proof}
By Lemma \ref{lemma: M is homeomorphic to the PIPs}, the set of PIPs is open in $\hat{M}$. Let $P = I^-(\gamma)$ be an TIP where $\gamma: [a,b[ \to M$ is a future-inextendible causal curve. Let $\{t_n\}$ be an increasing sequence of times in the interval $[a,b[$. We call $P_n$ the PIP defined as the chronological past of $\gamma(t_n)$. We prove that $\{P_n\}$ converges to $P$. By Remark \ref{remark: open neighborhoods of a TIP}, it is sufficient to prove that any open neighborhood of $P$ of the form $\mathcal{O}_U$ where $U = I^+(p)$ with $p \in M$, or of the form $\mathcal{O}'_U$ where $U = M \backslash J^+(p)$ with $p \in M$, contains all the $P_n$ for $n$ big enough.

Since $P_n \subset P$, it is clear that any open neighborhood $\mathcal{O}'_U$ of $P$ contains all the $P_n$. Let $\mathcal{O}_U$ be a neighborhood of $P$ where $U = I^+(p)$. Then, $p \in P$.  Since $\{t_n\}$ is increasing, for $n$ big enough, $p \in I^-(\gamma(t_n)) = P_n$. Hence, $P_n \cap I^+(p) \not = \emptyset$, i.e. $P_n \in \mathcal{O}_U$ for $n$ big enough. We deduce that $\{P_n\}$ converges to $P$.
\end{proof}

\begin{remark}
The same construction with the reverse time orientation defines a topology on the set $\check{M}$ of IFs. Propositions \ref{prop: topology on IPs is Hausdorff} and \ref{prop: PIPs are dense in IPs} still hold for $\check{M}$.
\end{remark}

Propositions \ref{prop: topology on IPs is Hausdorff} and \ref{prop: PIPs are dense in IPs} are the best we can say for a general globally hyperbolic spacetime. 

\section{Causal completion of developable GH conformally flat spacetimes} \label{sec: causal completion in the conformally flat setting}

We devote this section to the proof of our main result:

\begin{theorem} \label{thm: causal completion}
Let $M$ be a developable conformally flat GH spacetime of dimension $n \geq 3$ without conjugate points. The causal completion of $M$ is a topological manifold with boundary, homeomorphic to $S \times [0,1]$ where $S$ is any Cauchy hypersurface of $M$.
\end{theorem}

Notice that the assumptions \emph{developable} and \emph{without conjugate points} are necessary. Indeed, if one of them is not satisfied, we can easily point out counter-examples:
\begin{itemize}
\item The universal Einstein universe $\eeu$ is developable but contains conjugate points. Each of its future and past causal boundary is reduced to a single point. 
\item Consider the quotient of a regular domain $\Omega$ of $\R^{1,2}$ (in the sense of \cite{bonsante2005flat}) by a discrete subgroup of isometries $\Gamma$ of $\R^{1,2}$. In general, the action of $\Gamma$ on the singular points of the boundary of $\Omega$ in $\R^{1,2}$ (see \cite[Section 4]{bonsante2005flat}) is neither free nor properly discontinous. The quotient of the causal completion of $\Omega$ by $\Gamma$ is then far from being a topological manifold.
\end{itemize}

In Section \ref{sec: causally convex subsets of eeu}, we prove Theorem \ref{intro: main theorem} for causally convex open subsets of $\eeu$ before dealing with the general case in Section \ref{sec: general case}.

\subsection{The case of causally convex open subsets of Einstein universe} \label{sec: causally convex subsets of eeu}

Let $\Omega$ be a causally convex open subset of $\eeu$ without conjugate points. \\
By \cite[Proposition 3]{smai2023enveloping}, there exist two $1$-Lipschitz real-valued functions $f^+$ and $f^-$ defined on an open subset $U$ of $\mathbb{S}^{n-1}$ whose extensions to $\partial U$ coincide, such that
\begin{align*}
\Omega &= \{(x,t) \in U \times \R;\ f^-(x) < t < f^+(x)\}.
\end{align*}

\begin{proposition} \label{prop: causal boundary in ein}
The future (resp. past) causal boundary of $\Omega$ is homeomorphic to the graph of $f^+$ (resp. $f^-$).
\end{proposition}

\begin{proof}
Let $f$ be the map which associates to every point $p$ in the graph of $f^+$ the TIP $I^-(p) \cap \Omega$ of $\Omega$. We prove that $f$ is bijective. Let $p, q$ two points in the graph of $f^+$ such that $I^-(p) \cap \Omega = I^-(q) \cap \Omega$. Then, $I^-(I^-(p) \cap \Omega) = I^-(I^-(q) \cap \Omega)$. Hence, $I^-(p) = I^-(q)$. Since $\eeu$ is past-distinguishing, we deduce that $p = q$. Thus, $f$ is injective.

Now, let $P$ be a TIP of $\Omega$. Then, there exists an inextensible timelike curve $\gamma$ of $\Omega$ such that $P = I^-(\gamma, \Omega)$. Since $\Omega$ is causally convex, $\gamma$ is the intersection of an inextensible timelike curve $\tilde{\gamma}$ with $\Omega$. By \cite[Lemma 6]{smai2023enveloping}, $\tilde{\gamma}$ intersects the graph of $f^+$ exactly once in a point $p$. Thus, $P = I^-(p) \cap \Omega$. The map $f$ is then surjective.

We deduce that $f$ is bijective. It is easy to check that it is a homeomorphism.
\end{proof}

\begin{corollary}
The causal completion of $\Omega$ is a topological manifold with boundary homeomorphic to $S \times [0,1]$ where $S$ is a Cauchy hypersurface of $\Omega$. \qed
\end{corollary}

\subsection{The general case} \label{sec: general case}

Let $M$ be a developable GH conformally flat spacetime without conjugate points. Let $E(M)$ be \emph{an enveloping space} of $M$: this is a conformally flat developable spacetime which fibers trivially over a conformally flat Riemannian manifold $\mathcal{B}$ diffeomorphic to a Cauchy hypersurface of $M$, in which $M$ embeds conformally as a causally convex open subset $\Omega$ (see \cite[Theorem 1]{smai2023enveloping}). By \cite[Prop. 10]{smai2023enveloping}, $\Omega$ is the domain bounded by the graphs of two real-valued functions $f^+$ and $f^-$ defined on an open subset of $\mathcal{B}$: 
\begin{align*}
\Omega &= \{(x,t) \in U \times \R;\ f^-(x) < t < f^+(x)\}.
\end{align*}
Moroever, the graphs of $f^+$ and $f^-$ satisfy the property of being achronal in $E(M)$.

\begin{proposition} \label{prop: causal boundary general conf. flat setting}
The future (resp. past) causal boundary of $\Omega$ is homeomorphic to the graph of $f^+$ (resp. $f^-$).
\end{proposition}

\begin{proof}
Consider the map $f$ which associates to every point $p$ in the graph of $f^+$ the TIP $I^-(p) \cap \Omega$. We show that $f$ is bijective. Since $E(M)$ is developable, it is strongly causal (see \cite[Lemma 8]{smai2023enveloping}) and thus distinguishing (see \cite[Remark 3.23]{minguzzi2008causal}). Therefore, the same arguments used in the proof of Proposition \ref{prop: causal boundary in ein} show that $f$ is injective.

Let $P$ be a TIP of $\Omega$. Then, there exists an inextensible timelike curve $\gamma$ of $\Omega$ such that $P = I^-(\gamma)$. Let $p_0 \in \gamma$. Since $\Omega$ is GH, the restriction of $\hat{D}$ to $I^+(p_0, \Omega)$ is injective and its image is causally convex in $\eeu$ (see \cite[Prop. 2.7 and Cor. 2.8, p.151]{salveminithesis}). Clearly, $\gamma \cap I^+(p_0, \Omega)$ is a future-inextensible timelike curve of $I^+(p_0, \Omega)$, denoted $\gamma_0$. Therefore, by Proposition \ref{prop: causal boundary in ein}, $\gamma_0$ admits a future endpoint $p$ in the intersection of the graph of $f^+$ with $I^+(p_0, \Omega)$. Then, $P = I^-(p) \cap \Omega$. Thus, $f$ is surjective. It is easy to check that $f$ and its inverse are continuous. 
\end{proof}

Theorem \ref{thm: causal completion} follows immediately from Proposition \ref{prop: causal boundary general conf. flat setting}.

\section{$\mathcal{C}_0$-maximality and $\mathrm{C}$-maximality} \label{sec: maximality}

We devote this section to the proof of Theorem \ref{intro: C_0-maximality implies C-maximality}: we prove that any $\mathcal{C}_0$-maximal spacetime is $\mathrm{C}$-maximal. In Section \ref{sec: criterion of maximality}, we recall a criterion of $\mathrm{C}$-maximality involving the causal boundary. Then, in Section \ref{sec: C_0-maximality implies C-maximality}, we use this criterion to prove Theorem \ref{intro: C_0-maximality implies C-maximality}.

\subsection{Criterion of $\mathrm{C}$-maximality} \label{sec: criterion of maximality}

Let $f$ be a conformal Cauchy-embedding between two globally hyperbolic conformal spacetimes $M$ and $N$. 

\begin{proposition}[{\cite[Lemma 10]{Salvemini2013Maximal}}] \label{prop: boundary of the image of a Cauchy-embedding}
The boundary of $f(M)$ in $N$ is the union of two disjoint closed achronal egdeless sets (eventually empty) $\partial^+f(M)$ and $\partial^- f(M)$ such that
\begin{align*}
I^-(\partial^+ f(M)) \cap I^+(\partial^- f(M)) \subset f(M).
\end{align*}
\end{proposition}

There is subtle relation between the boundary of $f(M)$ in $N$ and the causal boundary of $M$:

\begin{proposition} \label{prop: Cauchy-embeddings and causal boundary}
Suppose $\partial^+ f(M)$ is non-empty. Then, for every point $p$ of $\partial^+ f(M)$, the pre-image under $f$ of the intersection of the chronological past of $p$ with $f(M)$ is a TIP of $M$.
\end{proposition}

There is a similar statement for $\partial^- f(M)$ with the reverse time-orientation.

\begin{proof}
Set $P = f^{-1}(I^-(p) \cap f(M))$. 
\begin{itemize}
\item \emph{$P$ is a past set, i.e. $P = I^-(P)$:} Since $P$ is open, $P \subset I^-(P)$. Conversely, let $x \in I^-(P)$. There exists $y \in P$ such that $x \in I^-(y)$. Then, since $f$ is conformal, $f(x) \in I^-(f(y))$. Moreover, since $y \in P$, we have $f(y) \in I^-(p)$. By transitivity, we get $f(x) \in I^-(p)$. Hence, $x \in P$. Thus, $I^-(P) \subset P$.
\item \emph{$P$ is indecomposable:} Suppose $P$ is the union of two distinct past sets $Q$ and $R$. Then, $f(P) = f(Q) \cup f(R)$, i.e. $I^-(p) \cap f(M) = f(Q) \cup f(R)$. Hence, $I^-(I^-(p) \cap f(M)) = I^-(f(Q) \cup f(R)) = I^-(f(Q)) \cup I^-(f(R))$. It is easy to see that $I^-(I^-(p) \cap f(M)) = I^-(p)$. Thus, $I^-(p) = I^-(f(Q)) \cup I^-(f(R))$. Since $I^-(p)$ is indecomposable, we deduce that $I^-(p)= I^-(f(Q))$ or $I^-(p) = I^-(f(R))$. Without loss of generality, we suppose that $I^-(p) = I^-(f(Q))$. Thus, $I^-(p) \cap f(M) = I^-(f(Q)) \cap f(M)$. Since $f(M)$ is causally convex, the intersection $I^-(f(Q)) \cap f(M)$ is exactly the chronological past of $f(Q)$ in $f(M)$. But $f(Q)$ is a past set in $f(M)$. Thus, $I^-(f(Q)) \cap f(M) = f(Q)$. Hence, $I^-(p) \cap f(M) = f(Q)$. It follows that $P = Q$. In other words, $P$ is indecomposable.
\item \emph{$P$ is terminal:} Suppose there exists $q \in M$ such that $P = I^-(q)$. Then, we have $I^-(p) \cap f(M) = f(I^-(q))$. Since $f(M)$ is causally convex in $N$ (see \cite[Lemma 8]{Salvemini2013Maximal}), $f(I^-(q)) = I^-(f(q)) \cap f(M)$. Hence, $I^-(p) \cap f(M) = I^-(f(q)) \cap f(M)$. Then, $I^-(I^-(p) \cap f(M)) = I^-(f(q) \cap f(M))$. Hence, $I^-(p) = I^-(f(q))$. Since $M$ is globally hyperbolic, it follows that $p = f(q)$. Contradiction.
\end{itemize}
\end{proof}

\begin{corollary} \label{cor: Cauchy embedding and causal boundary}
The map which associates to every point $p$ of $\partial^+ f(M)$ the TIP of $M$ defined by $f^{-1}(I^-(p) \cap f(M))$ is injective.
\end{corollary}

\begin{proof}
Let $p, q \in \partial^+ f(M)$ such that $f^{-1}(I^-(p) \cap f(M)) = f^{-1}(I^-(q) \cap f(M))$. Then, $I^-(p) \cap f(M) = I^-(q) \cap f(M)$. Thus, $I^-(I^-(p) \cap f(M)) = I^-(I^-(q) \cap f(M))$, i.e. $I^-(p) = I^-(q)$. Hence, $p = q$.
\end{proof}

A consequence of the description above is the following criterion of $\mathrm{C}$-maximality.

\begin{proposition}[Criterion of maximality] \label{criterion of maximality}
Let $M$ be a globally hyperbolic conformal spacetime. Suppose that $M$ admits a non-compact Cauchy hypersurface $S$. Then, if the intersection of $S$ with any TIP and any TIF of $M$ is non-compact, $M$ is $\mathrm{C}$-maximal.
\end{proposition}

\begin{proof}
Let $f$ be a conformal Cauchy embedding from $M$ to a globally hyperbolic spacetime $N$. Let $S$ be a Cauchy hypersurface of $M$. Suppose that $f$ is not surjective. Then $\partial^+ f(M)$ (or $\partial^- f(M)$) is non-empty. Let $p \in \partial f^+(M)$. By Proposition \ref{prop: Cauchy-embeddings and causal boundary}, $f^{-1}(I^-(p) \cap f(M))$ is a TIP of $M$. The intersection of this TIP with $S$ is not compact. But, its image, equal to $I^-(p) \cap f(S)$ is compact. Contradiction. The proposition follows.
\end{proof}

\subsection{$\mathcal{C}_0$-maximality implies $\mathrm{C}$-maximality} \label{sec: C_0-maximality implies C-maximality}

Let $M$ be a globally hyperbolic conformally flat spacetime. We suppose that $M$ is $\mathcal{C}_0$-maximal. In what follows, we prove that $M$ is $\mathrm{C}$-maximal.\\

By \cite[Corollary 6]{smai2023enveloping}, the universal cover of $M$, denoted $\tilde{M}$, is $\mathcal{C}_0$-maximal. If $\tilde{M}$ is Cauchy-compact, it is conformally equivalent to the universal cover of Einstein universe (see \cite[Theorem 9]{Salvemini2013Maximal}).

\begin{proposition}
The spacetime $\eeu$ is $\mathrm{C}$-maximal.
\end{proposition}

\begin{proof}
Let $f$ be a conformal Cauchy-embedding from $\eeu$ in a globally hyperbolic spacetime $N$. Suppose that $f$ is not surjective. Then, $\partial^+ f(M)$ (or $\partial^- f(M)$) is non-empty. By Corollary \ref{cor: Cauchy embedding and causal boundary}, there is an injection from $\partial^+ f(M)$ to the future causal boundary of $\eeu$. This last one is reduced to a single point. Thus, $\partial^+ f(M)$ is equal to this point. But, $\partial^+ f(M)$ is edgeless (see Proposition \ref{prop: boundary of the image of a Cauchy-embedding}). Contradiction. Hence, $\eeu$ is $\mathrm{C}$-maximal.  
\end{proof}

Now, suppose that $\tilde{M}$ admits a non-compact Cauchy hypersurface. Let $E(\tilde{M})$ be an enveloping space of $\tilde{M}$; it fibers trivially over a conformally flat Riemannian manifold $\mathcal{B}$. The spacetime $\tilde{M}$ embeds conformally in $E(\tilde{M})$ as a causally convex open subset $\Omega$. Let $f^+, f^-$ two $1$-Lipschitz real-valued functions defined on an open subset $U$ of $\mathcal{B}$ such that
\begin{align*}
\Omega &= \{(x,t) \in U \times \R;\ f^-(x) < t < f^+(x)\}.
\end{align*} 
The $\mathcal{C}_0$-maximality of $\Omega$ is characterized by the following property of the graphs of $f^+$ and $f^-$ (see \cite[Prop. 18]{smai2023enveloping}):

\begin{fact}\label{fact: eikonal}
For every point $p$ in the graph of $f^+$, there exists a past-directed lightlike geodesic starting from $p$, entirely contained in the graph of $f^+$ and with no past endpoint in the graph of $f^+$.
\end{fact}

\begin{proposition} \label{prop: C_0-maximality implies C-maximality simply connected case}
The causally convex open subset $\Omega$ of $E(\tilde{M})$ is $\mathrm{C}$-maximal.
\end{proposition}

\begin{proof}
We use the criterion given by Proposition \ref{criterion of maximality}. Let $S$ be a Cauchy hypersurface of $\Omega$. Let $P$ be a TIP of $\Omega$. By Proposition \ref{prop: causal boundary general conf. flat setting}, there exists a unique point $p$ in the graph of $f^+$ such that $P = I^-(p) \cap \Omega$. By Fact \ref{fact: eikonal}, there exists a past-directed lightlike geodesic starting from $p$, entirely contained in the graph of $f^+$. Therefore, $P \cap S$ is not compact. Similarly, the intersection of any TIF of $\Omega$ with $S$ is not compact. Thus, $M$ is $\mathrm{C}$-maximal.
\end{proof}

\begin{proof}[Proof of Theorem \ref{intro: C_0-maximality implies C-maximality}]
Let $f$ be a conformal Cauchy-embedding from $M$ to a globally hyperbolic spacetime $N$. It is easy to see that the lift $\tilde{f}$ of $f$ is a Cauchy-embedding from $\tilde{M}$ to $\tilde{N}$. It follows from Proposition \ref{prop: C_0-maximality implies C-maximality simply connected case} that $\tilde{f}$ is surjective. Thus, $f$ is surjective. In other words, $M$ is $\mathrm{C}$-maximal.
\end{proof}

From now on, by Theorem \ref{intro: C_0-maximality implies C-maximality}, we simply say that a conformally flat spacetime is \emph{maximal} (while keeping in mind that it is for the ordering relation defined by \emph{conformal} Cauchy-embeddings).

\section{Complete photons} \label{sec: examples}

In \cite{Salvemini2013Maximal}, Rossi proved that any developable GHM conformally flat spacetime containing conjugate points is conformaly equivalent to $\eeu$. As a consequence, any GHM conformally flat spacetime whose universal cover admits conjugate points is a finite quotient of $\eeu$. 

In this section, we introduce three remarkable families of GHM conformally flat spacetimes obtained as the quotient of a causally convex open subset $\Omega \subsetneq \eeu$ by a discrete group of conformal transformations of $\eeu$. Rossi's result insures that the domains $\Omega$ do not contain conjugate points. Nevertheless, we prove that they satisfy the remarkable property of containing \emph{complete} photons:

\begin{definition} \label{def: complete photons}
Let $M$ be a developable conformally flat spacetime. A photon of $M$ is said \textbf{complete} if it develops on a segment of photon of $\eeu$ connecting two conjugate points.
\end{definition}

The GHCM conformally flat spacetimes that we describe typically arise from the data of some appropriate \mbox{\emph{$P_1$-Anosov representation}} of a Gromov hyperbolic group into $O(2,n)$.

\subsection{Anosov representations and GHCM conformally flat spacetimes} 

We recall here the notion of $P_1$-Anosov representation in $O_0(2,n)$ and how they produce examples of GHCM conformally flat spacetimes.

Let $\Gamma$ be a Gromov hyperbolic group and let $\rho$ be a representation of $\Gamma$ in $O_0(2,n)$. The notion of $P_1$-Anosov representation involves two dynamics. On the one hand, the north-south dynamic on the Gromov boundary of $\Gamma$, denoted $\partial_{\infty} \Gamma$; and on the other hand, the dynamic on $Ein_{1,n-1}$ under the action of some sequences of $O(2,n)$, namely \emph{$P_1$-divergent sequences}. A sequence $\{g_i\}$ of $O(2,n)$ is said \emph{$P_1$-divergent} if it is divergent - i.e. leaves every compact set - and if there exist a subsequence $\{g_{j_i}\}$, an attracting point $p_+$ and a repulsing point $p_-$ in $Ein_{1,n-1}$ such that $\{g_{i_j}\}$ converges uniformly to $p_+$ on every compact set disjoint from the lightcone of $p_-$.

\begin{definition} 
The representation $\rho: \Gamma \to O_0(2,n)$ is said $P_1$-Anosov if
\begin{enumerate}
\item all sequences of $\rho(\Gamma)$ are $P_1$-divergent;
\item there exists a continuous $\rho$-equivariant map $\xi: \partial_{\infty} \Gamma \to Ein_{1,n-1}$ which is
\begin{enumerate}
\item \emph{transverse} meaning that for every pair $(\eta, \eta')$ of distinct points of $\partial_{\infty} \Gamma$, the points $\xi(\eta)$ and $\xi(\eta')$ are not connected by a lightlike geodesic;
\item \emph{dynamics-preserving} meaning that if $\eta \in \partial_{\infty} \Gamma$ is an attracting point of $\gamma \in \Gamma$, then $\xi(\eta) \in Ein_{1,n-1}$ is an attracting point of $\rho(\gamma)$.
\end{enumerate}
\end{enumerate} 
\end{definition}

Any $P_1$-Anosov representation preserves a specific closed subset of $Ein_{1,n-1}$ called \emph{limit~set}: 
\begin{itemize}
\item \emph{The limit set of $\rho$}, denoted $\Lambda_{\rho}$, is the set of all attracting points of the elements of~$\rho(\Gamma)$. It coincides with the image of the boundary map $\xi: \partial_{\infty} \Gamma \to Ein_{1,n-1}$.
\item We say that the representation $\rho$ is \emph{negative} if the limit set $\Lambda_{\rho}$ lifts to an \emph{acausal} subset of $\eeu$.
\end{itemize}

In a previous paper, we proved the following result.

\begin{theorem}[{\cite[Theorem 5.1]{smai2022anosov}}]
Any $P_1$-Anosov representation $\rho$ of a Gromov hyperbolic group $\Gamma$ in $O_0(2,n)$ with \emph{negative} limit set which is not a topological $(n-1)$-sphere, is the holonomy of a GHCM conformally flat spacetime $M_{\rho}$ of dimension $n$.
\end{theorem}

The spacetime $M_{\rho}$ is constructed as follow. We define \emph{the invisible domain} $\Omega(\Lambda_{\rho})$ of the limit set $\Lambda_{\rho}$:
\begin{align*}
\Omega(\Lambda_{\rho}) &:= \{\xi \in Ein_{1,n-1}:\ <\xi, \xi_0>_{2,n} < 0\ \forall \xi_0 \in \Lambda_{\rho}\}.
\end{align*}
In the terminology of \cite[Def. 9]{smai2023enveloping}, $\Omega(\Lambda_{\rho})$ is \emph{the dual} of $\Lambda_{\rho}$. It corresponds to the set of points of $Ein_{1,n-1}$ which are not causally related to any point of the limit set in the following sense: 

Let $\tilde{\Lambda}_{\rho}$ be a lift of $\Lambda_{\rho}$ in $\eeu$. We call $\Omega(\tilde{\Lambda}_{\rho})$ the set of points of $\eeu$ which are not causally related to any point of $\tilde{\Lambda}_{\rho}$:
\begin{align*}
\Omega(\tilde{\Lambda}_{\rho}) &= \eeu \backslash (J^+(\Lambda_{\rho}) \cup J^-(\Lambda_{\rho})). 
\end{align*}
The restriction of the projection $\pi: \eeu \to Ein_{1,n-1}$ to $\Omega(\tilde{\Lambda}_{\rho})$ is injective and its image is exactly $\Omega(\Lambda_{\rho})$ (see \cite[Lemmas 5.11 \& 5.12]{smai2022anosov}).

We prove that the action of $\rho(\Gamma)$ on $\Omega(\Lambda)$ is free and properly discontinuous (see \cite[Sec. 5.2]{smai2022anosov}). The spacetime $M_{\rho}$ is then the quotient of $\Omega(\Lambda_{\rho})$ by $\rho(\Gamma)$ (see \cite[Sec. 5.3]{smai2022anosov}).

\subsection{GHCM conformally flat spacetimes with complete photons}

Throughout this section, $\Lambda$ denotes a closed negative subset of $Ein_{1,n-1}$ - which typically arises as the limit set of some $P_1$-Anosov representation in $O(2,n)$. We describe the invisible domain $\Omega(\Lambda)$ in the following cases:
\begin{enumerate}
\item$\Lambda$ is contained in a Penrose boundary $\mathcal{J}^+(\mathrm{x})$ where $\mathrm{x} \in Ein_{1,n-1}$, defining \emph{black-white holes};
\item $\Lambda$ is a conformal sphere of dimension $0 \leq k \leq n-3$, defining \emph{Misner domains of Einstein universe};
\item $\Lambda$ is strictly contained in a conformal sphere of dimension $0 \leq k \leq n-3$, defining \emph{Misner extensions}.
\end{enumerate}

\subsubsection{Black-white holes} \label{sec: black holes}

Let $\mathrm{x} \in Ein_{1,n-1}$. Suppose that $\Lambda$ is contained in $\mathcal{J}^+(\mathrm{x})$ and that it is not a topological $(n-2)$-sphere (i.e. it is not a section of $\mathcal{J}^+(\mathrm{x})$).

\paragraph{Horizons.} The intersection of $\Omega(\Lambda)$ with $\mathcal{J}^+(\mathrm{x})$ is the union of the lightlike geodesics with extremities $\mathrm{x}$ and $-\mathrm{x}$ disjoint from $\Lambda$. Every connected component of this union is called \textbf{a horizon}. The assumption that $\Lambda$ is not a topological $(n-2)$-sphere insures the existence of horizons. By definition, each horizon is foliated by complete photons.

\paragraph{Black hole, white hole.} $\mathcal{J}^+(\mathrm{x})$ is the future Penrose boundary of the affine chart $M(\mathrm{x})$ and the past Penrose boundary of $M(-\mathrm{x})$. Hence,
\begin{enumerate}
\item the intersection of $\Omega(\Lambda)$ with $M(\mathrm{x})$ is a future-regular domain of $M(\mathrm{x})$, denoted $\Omega^+(\Lambda)$;
\item the intersection of $\Omega(\Lambda)$ with $M(-\mathrm{x})$ is a past-regular domain of $M(-\mathrm{x})$, denoted $\Omega^-(\Lambda)$.
\end{enumerate}

The domain $\Omega^-(\Lambda)$ satisfies the property that no photon of $\Omega(\Lambda)$ going through a point of $\Omega^-(\Lambda)$ can escape from $\Omega^-(\Lambda)$ in the future, it can only escape through one of the horizons in the past. The domain $\Omega^+(\Lambda)$ satisfies a similar property for the reverse time-orientation. For this reason, $\Omega^-(\Lambda)$ is called \emph{a black hole} and $\Omega^+(\Lambda)$ \emph{a white-hole}. 

\begin{figure}[h!]
\centering
\includegraphics[scale=0.9]{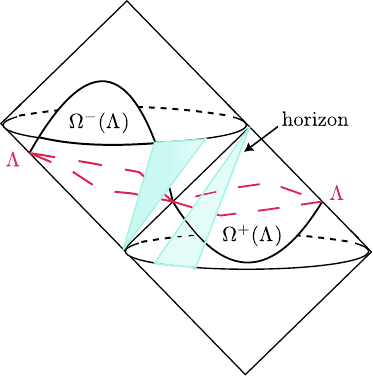}
\caption{Black-white hole domain of $Ein_{1,2}$.}
\end{figure}

\begin{definition}
The invisible domain $\Omega(\Lambda)$ is called \emph{a black-white hole domain of $Ein_{1,n-1}$}. The points $\mathrm{x}$ and $-\mathrm{x}$ are \emph{the endpoints} of the horizons of $\Omega(\Lambda)$.
\end{definition}

\paragraph{Causal boundary.} The future causal boundary of $\Omega(\Lambda)$ is the boundary of $\Omega^-(\Lambda)$ in $M(-\mathrm{x})$ while its past causal boundary  is the boundary of $\Omega^+(\Lambda)$ in $M(\mathrm{x})$.

\paragraph{Black-white holes conformally flat spacetimes.} Let $\Gamma$ be a discrete group of $O_0(2,n)$ preserving $\Omega(\Lambda)$ such that
\begin{enumerate}
\item the action of $\Gamma$ on $\Omega(\Lambda)$ is free and properly discontinuous; \label{cond. 1}
\item $\Gamma$ fixes $\mathrm{x}$. \label{cond. 2}
\end{enumerate}
The point \ref{cond. 1}. insures that the quotient $M := \Omega(\Lambda) \backslash \Gamma$ is a conformally flat spacetime.
The point \ref{cond. 2}. insures that $\Gamma$ preserves the decomposition of $\Omega(\Lambda)$ as the disjoint union of the black hole $\Omega^-(\Lambda)$, the horizons $H_i$, and the white hole $\Omega^+(\Lambda)$. As a consequence, the conformally flat spacetime $M$ can be written as the disjoint union of the black hole $\mathcal{B} := \Omega^-(\Lambda) \backslash \Gamma$, the horizons $\mathcal{H}_i := H_i \backslash \Gamma$, and the white hole $\mathcal{W} := \Omega^+(\Lambda) \backslash \Gamma$. This motivates the following definition.

\begin{definition}
We call \emph{black-white hole} any GHCM conformally flat spacetime obtained as the quotient of a black-white hole domain $\Omega(\Lambda)$ of $Ein_{1,n-1}$ by a discrete subgroup of $O(2,n)$ fixing the endpoints of the horizons of $\Omega(\Lambda)$.
\end{definition}

\begin{example}
The data of a negative $P_1$-Anosov representation $\rho$ of a Gromov hyperbolic group $\Gamma$ in $O_0(2,n)$ fixing a point $\mathrm{x} \in Ein_{1,n-1}$ such that the limit set $\Lambda_{\rho}$
\begin{enumerate}
\item is contained in $\mathcal{J}^+(\mathrm{x})$,
\item is not a topological $(n-2)$-sphere;
\end{enumerate}
defines a black-while hole $\Omega(\Lambda_{\rho}) \backslash \rho(\Gamma)$ (see \cite[Theorem 5.1]{smai2022anosov}).
\end{example}

\subsubsection{Conformally flat Misner spacetimes} 

Suppose that $\Lambda$ is a $(\ell - 1)$-conformal sphere where $\ell \in \mathbb{N}^*$ such that $\ell < n - 1$. It~is defined by the data of a Lorentzian subspace $\R^{1,\ell} \subset \R^{2,n}$. Consider the orthogonal splitting $\R^{2,n} = \R^{1,\ell} \oplus^\perp \R^{1,k}$ where $k \in \mathbb{N}^*$ such that $k + \ell = n$. We denote $q_{1,\ell}$ and $q_{1,k}$ the restrictions of the quadratic form $q_{2,n}$ to $\R^{1,\ell}$ and $\R^{1,k}$ respectively; notice that $q_{2,n} = q_{1,\ell} + q_{1,k}$. The sphere $\Lambda$ is a connected component of the projectivization of the quadric 
\begin{align*}
\{(x,0): \in \R^{1,\ell} \oplus \R^{1,k}:\ q_{1,\ell}(x) = 0\}.
\end{align*}
Moreover, the projectivization of the quadric
\begin{align*}
\{(0,y): \in \R^{1,\ell} \oplus \R^{1,k}:\ q_{1,k}(y) = 0\}
\end{align*}
is the disjoint union of two antipodal conformal $(k-1)$-spheres $\mathbb{S}^{k-1}_+$ and $\mathbb{S}^{k-1}_-$ of $Ein_{1,n-1}$. Notice that $\Lambda$ is a connected component of the intersection of the lightcones of the points of $\mathbb{S}^{k-1}_+$.

\paragraph{Homogeneous model of the invisible domain.} Let $\mathcal{C}_\ell$ be the connected component of the causal cone of $\R^{1,\ell}$ defining $\Lambda$. In other words, $\Lambda$ is the projectivization of the boundary of $\mathcal{C}_{\ell}$ in $\R^{1,\ell}$. Let
\begin{align*}
\hy^\ell &= \{x \in \mathcal{C}_\ell:\ q_{1,\ell}(x) = -1\}
\end{align*}
be the hyperbolic space of dimension $\ell$ and let
\begin{align*}
dS_{1,k-1} &= \{y \in \R^{1,k}:\ q_{1,k}(y) = 1\}
\end{align*}
be the de Sitter space of dimension $k$.

\begin{proposition}
The invisible domain $\Omega(\Lambda)$ is conformally equivalent to the homogeneous space $\hy^\ell \times dS_{1,k-1}$.
\end{proposition}

\begin{proof}
We denote by $<.,.>_{1,\ell}$ the bilinear form on $\R^{1,\ell}$ associated to the quadratic form $q_{1,\ell}$. Let $[x:y] \in Ein_{1,n-1}$. We have
\[\begin{array}{cccc}
[x:y] \in \Omega(\Lambda) & \Leftrightarrow & <(x;y),(z;0)>_{2,n} < 0, & \forall [z:0] \in \Lambda \\
                          & \Leftrightarrow & <x,z>_{1,\ell} < 0, & \forall z \in \partial \mathcal{C}_{\ell}
\end{array}\]
where $\partial \mathcal{C}_{\ell}$ denotes the boundary of $\mathcal{C}_{\ell}$ in $\R^{1,\ell}$. This last condition can be rephrase by saying that $x$ is in the intersection of the strict future half-spaces of $\R^{1,\ell}$ bounded by the degenerate hyperplanes $<.,z>_{1,\ell} = 0$ where $z \in \partial \mathcal{C}_{\ell}$. It is easy to see that this intersection is exactly the cone $\mathcal{C}_{\ell}$. Hence, $[x:y] \in \Omega(\Lambda)$ if and only if $x \in \mathcal{C}_{\ell}$. Up to rescaling, one can suppose that $q_{1,\ell}(x) = -1$, i.e. $x \in \hy^\ell$. Since \mbox{$0 = q_{2,n}(x;y) = q_{1,\ell}(x) + q_{1,k}(y)$,} we deduce that $q_{1,k}(y) = 1$, i.e. $y \in dS_{1,k-1}$. The proposition follows.
\end{proof}

\paragraph{Black-white hole decomposition.} Fix a point $\mathrm{x} \in \mathbb{S}^{k-1}_-$. Then, the sphere $\Lambda$ is contained in $\mathcal{J}^+(\mathrm{x})$. Hence, the description presented in Section \ref{sec: black holes} still holds here:

The invisible domain $\Omega(\Lambda)$ is the disjoint union of a black hole $B$, horizons $H_i$ foliated by complete photons and a white hole $W$. The fact that $\Lambda$ is a conformal $(\ell - 1)$-sphere implies that $B$ and $W$ are past and future Misner domains (see Lemma \ref{lemma: euclidean hyperplanes in an affine chart}). Moreover, there is a single horizon when $\ell < n - 1$ and there are exactly two horizons when $\ell = n - 2$ (see Figure \ref{figure: Misner domain}).

\begin{figure}[h!]
\centering
\includegraphics[scale=0.9]{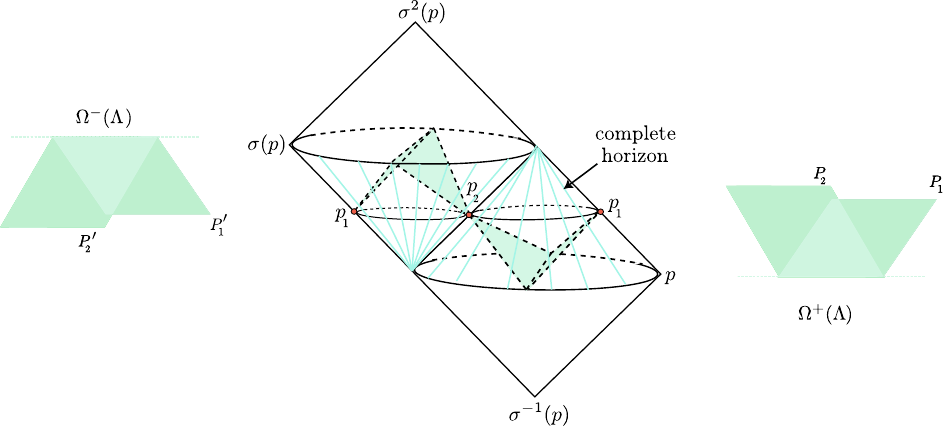}
\caption{Black-white hole decomposition of a Misner domain of $Ein_{1,2}$. The sphere $\Lambda$ is of dimension $0$: this is the union of two non-causally related points $p_1$ and $p_2$ of $Ein_{1,2}$.}
\label{figure: Misner domain}
\end{figure}

In short, every point of $\mathbb{S}^{k - 1}_+$ defines a decomposition of the invisible domain as the union of two Misner domains separated by one or two horizons. Since $\mathbb{S}^{k-1}_-$ is antipodal to $\mathbb{S}^{k-1}_+$, this statement also holds for the points of $\mathbb{S}^{k-1}_+$. This motivates the following definition.

\begin{definition}
We call \emph{Misner domain of Einstein universe} any open subset of the form $\Omega(\Lambda)$ where $\Lambda$ is a conformal $(\ell - 1)$-sphere with $0 < \ell < n - 1$.
\end{definition}

\paragraph{Causal boundary.} The future/past causal boundary of $\Omega(\Lambda)$ is the union of $\mathbb{S}^{k-1}_\pm$ and the union of the lightlike geodesics joining a point of $\mathbb{S}^{k-1}_\pm$ to a point of $\Lambda$. The set of maximal TIPs is exactly the sphere $\mathbb{S}^{k - 1}_+$ and the set of maximal TIFs is exactly the sphere $\mathbb{S}^{k-1}_+$ (see Figure \ref{figure: Misner domain}).

\paragraph{Conformally flat Misner spacetimes.} Let $\Gamma$ be a subgroup preserving the conformal sphere $\Lambda$. \emph{A priori}, $\Gamma$ does not fix $\Lambda$ point by point, so the black-while decomposition of $\Omega(\Lambda)$ described above is not preserved by $\Gamma$. However, since $\Gamma$ preserves $\Lambda$, it preserves $\mathbb{S}^k_\pm$. Hence, every element $\gamma \in \Gamma$ sends a black-while hole decomposition of $\Omega(\Lambda)$ on a conformally equivalent black-white hole decomposition. 

\begin{definition}
We call \emph{conformally flat Misner spacetime} any GHCM conformally flat spacetime obtained as the quotient of a Misner domain $\Omega(\Lambda)$ of Einstein universe by a discrete subgroup $\Gamma$ of $O(2,n)$ preserving $\Lambda$.
\end{definition}

\begin{example}
Let $O_0(1,\ell) \times O_0(1,k)$ the subgroup of $O_0(2,n)$ preserving the splitting $\R^{2,n} = \R^{1,\ell} \oplus \R^{1,k}$. Consider a negative $P_1$-Anosov representation $\rho$ of a Gromov hyperbolic group $\Gamma$ in $O_0(2,n)$ defined by a pair $(\rho_\ell, \rho_k)$ where
\begin{enumerate}
\item $\rho_\ell: \Gamma \to O_0(1,\ell)$ is a cocompact representation whose limit set is the conformal $(\ell - 1)$-sphere $\Lambda$;
\item $\rho_k: \Gamma \to O_0(1,k)$ is a relatively compact representation, i.e. the image $\rho_k(\Gamma)$ is contained in a compact of $O_0(1,k)$.
\end{enumerate}
The limit set $\Lambda_\rho$ of the representation $\rho = (\rho_\ell, \rho_k)$ is the conformal $(\ell - 1)$-sphere $\Lambda$. Hence, the quotient $\Omega(\Lambda_{\rho}) \backslash \rho(\Gamma)$ is a conformally flat Misner spacetime.
\end{example}

\paragraph{Extensions of conformally flat Misner spacetimes.} We ask the following question:
\begin{center}
\emph{Is there a discrete subgroup $\Gamma$ of $O_0(2,n)$ preserving the sphere $\Lambda$ and a causally convex open subset $\Omega'$ of $Ein_{1,n-1}$ containing strictly $\Omega(\Lambda)$?}
\end{center}
The answer to this question is yes! Anosov representations give examples of such subgroups $\Gamma$. Indeed, let $\rho$ be a negative $P_1$-Anosov representation of a Gromov hyperbolic group $\Gamma$ in $O_0(2,n)$ defined by a pair $(\rho_\ell, \rho_k)$ where
\begin{enumerate}
\item $\rho_\ell: \Gamma \to O_0(1,\ell)$ is a convex cocompact representation such that the limit set $\Lambda_{\rho_\ell}$ is \emph{strictly} contained in the conformal $(\ell - 1)$-sphere $\Lambda$;
\item $\rho_k: \Gamma \to O_0(1,k)$ is a relatively compact representation, i.e. the image $\rho_k(\Gamma)$ is contained in a compact of $O_0(1,k)$.
\end{enumerate}
The limit set $\Lambda_\rho$ of the representation $\rho$ is exactly $\Lambda_{\rho_\ell}$. Since $\Lambda_\rho \subsetneq \Lambda$, the invisible domain $\Omega(\Lambda_\rho)$ contains stricly $\Omega(\Lambda)$. Moreover, since $\rho(\Gamma)$ preserves $\Lambda$, it preserves $\Omega(\Lambda)$. \mbox{Since $\Omega(\Lambda) \subset \Omega(\Lambda_{\rho})$,} the action of $\rho(\Gamma)$ on $\Omega(\Lambda)$ is free and properly discontinuous. The quotient spacetime $M_\rho = \Omega(\Lambda_\rho) \backslash \rho(\Gamma)$ is then a GHCM conformally flat extension of the conformally flat Misner spacetime $\Omega(\Lambda) \backslash \rho(\Gamma)$. Extensions of conformally flat Misner spacetimes form a larger class of GHCM conformally flat spacetimes with complete photons.

\bibliographystyle{plain}
\bibliography{Biblio}

\end{document}